\documentclass[11pt]{article}

\usepackage{todonotes}

\usepackage[top=3.5cm,bottom=3.5cm,left=3.5cm,right=3cm]{geometry}
\pdfpagewidth\paperwidth
\pdfpageheight\paperheight
\usepackage[english,italian]{babel}
\usepackage[T1]{fontenc}
\usepackage[latin1]{inputenc} 
\usepackage{lmodern}
\usepackage{amsfonts}
\usepackage{cancel}
\usepackage{amsmath}
\usepackage{amsthm}
\usepackage{amssymb}
\usepackage{graphicx}
\usepackage{sidecap}
\usepackage{caption}
\usepackage{subfig}
\usepackage{wrapfig}
\usepackage{psfrag}
\usepackage{mathrsfs}
\usepackage{tikz}
\usepackage{multicol}
\usepackage{pgfplots}
\usepackage{hyperref,enumitem}

\newtheorem{theorem}{Theorem}
\newtheorem{lemma}[theorem]{Lemma}
\newtheorem{proposition}[theorem]{Proposition}

\theoremstyle{definition}
\newtheorem{definition}[theorem]{Definition}

\theoremstyle{remark}

\newtheorem*{remark*}{Remark}

\newcommand{\D}{{\rm d}}
\newcommand{\FBN}{\mathbb N}
\newcommand{\FBR}{\mathbb R}
\newcommand{\FBC}{\mathbb C}

\newcommand{\FBcD}{{\mathcal D}}
\newcommand{\FBcE}{{\mathcal E}}
\newcommand{\FBcF}{{\mathcal F}}
\newcommand{\FBcO}{{\mathcal O}}
\newcommand{\FBcR}{{\mathcal R}}
\newcommand{\FBcS}{{\mathcal S}}
\newcommand{\FBcW}{{\mathcal W}}

\newcommand{\email}[1]{\protect\href{mailto:#1}{#1}}

\begin{document}
\selectlanguage{english}

\title{The Shearlet Transform and Lizorkin Spaces}
%\titlerunning{The Shearlet transform and Lizorkin spaces}
% Use \titlerunning{<name>} for an abbreviated version of
\author{Francesca Bartolucci \thanks{Department of Mathematics, ETH Zurich, Raemistrasse 101, 8092 Zurich, Switzerland (\email{francesca.bartolucci@sam.math.ethz.ch}).}
\and  Stevan Pilipovi\'c \thanks{Department of Mathematics and Informatics, Faculty of Sciences, University of Novi Sad, Trg Dositeja Obradovi\' ca 4, 21000 Novi Sad, Serbia (\email{stevan.pilipovic@gmail.com}, \email{nenad.teofanov@dmi.uns.ac.rs}).} 
\and Nenad Teofanov \footnotemark[2]}
%\author{Francesca Bartolucci, Stevan Pilipovi\'c and Nenad Teofanov}
%% Use \authorrunning{Short Title} for an abbreviated version of
%% your contribution title if the original one is too long
%\institute{Francesca Bartolucci \at ETH Zurich, Department of Mathematics, Raemistrasse 101, 8092 Zurich, Switzerland, \email{francesca.bartolucci@sam.math.ethz.ch}
%\and Stevan Pilipovi\'c \at University of Novi Sad, Faculty of Sciences, Department of Mathematics and Informatics, Trg Dositeja
%Obradovi\' ca 4, 21000 Novi Sad, Serbia \email{stevan.pilipovic@gmail.com}
%\and Nenad Teofanov \at University of Novi Sad, Faculty of Sciences, Department of Mathematics and Informatics, Trg Dositeja
%Obradovi\' ca 4, 21000 Novi Sad, Serbia \email{nenad.teofanov@dmi.uns.ac.rs}}
\maketitle

\abstract{We prove a continuity result for the shearlet transform when restricted to the space of smooth and rapidly decreasing functions
with all vanishing moments.  We define the dual shearlet transform, called here the shearlet synthesis operator, and we prove its continuity on the space of smooth and rapidly decreasing functions over $\FBR^2\times\FBR\times\FBR^{\times}$. Then, we use these continuity results to extend the shearlet transform to the space of Lizorkin distributions, and we prove its consistency with the classical definition for test functions.
}\\\\

\vspace{2mm}

\section{Introduction}
\label{FBsec:introduction}
Among the large reservoir of directional multiscale representations which have been introduced over the years, the shearlet representation has gained considerable attention for its capability to resolve the
wavefront set of distributions, providing both the location and the
geometry of the singularity set of signals. Indeed, when we shift from one-dimensional to multidimensional signals, it is not just of interest
to locate singularities in space but also to describe how they are distributed. This additional information is expressed by the notion of wavefront set introduced by H\"ormander in \cite{FBhormander83}.
In \cite{FBkula09} the authors
show that the decay rate of the shearlet
coefficients $\FBcS_{\psi}f(b,s,a)$ of a signal $f$ with respect to suitable
shearlets $\psi$ characterizes the wavefront set of $f$. Precisely, they show that for any signal $f\in L^2(\FBR^2)$ the shearlet coefficients $\FBcS_{\psi}f(b,s,a)$ exhibit fast asymptotic decay as $a\to0$ except when the pair $(b,(\xi_1,\xi_2))\in\FBR^2\times\FBR^2$, with $\xi_2/\xi_1=s$, belongs to the wavefront set of $f$. Later this result has
been generalized in \cite{FBgr11} where it is shown that
the same result holds true under much weaker assumptions on the
admissible vectors by means of a new approach based on an adaptation of the Radon transform to the shearlet structure, the affine Radon transform.

\par

On the other hand, while the classical wavelet transform is widely
exploited in signal analysis for describing pointwise smoothness
of univariate functions (we refer to \cite{FBjaffard89,FBmallat09} as classical references), it has proved not flexible enough to capture the geometry of the singularity set when we shift from one-dimensional signals to multidimensional signals.
We refer to e.g. \cite{FBpilvul} for a modification of the wavelet transform which overcomes these difficulties.

\par
In some sense, shearlets behave for high-dimensional signals as wavelets do for one-dimensional signals and the link between these two transforms has been clarified in \cite{FBbardemadeviodo} where it is shown that the shearlet transform is the composition of the affine Radon transform with a one-dimensional wavelet transform, followed by a convolution operator with a scale-dependent filter.
\par

There are at least two classical approaches to extend integral transforms to generalized function spaces. The coorbit space theory introduced by Feichtinger and Gr\"ochenig in \cite{feichtingergrochenig89,feichtgroche89} applies when the integral transform is the voice transform associated to a square-integrable representation of a locally compact group, and this is the case of the shearlet transform. We refer to \cite{dahlikeetal} for an extension of the shearlet transform based
on the coorbit space theory. The second way to proceed is the duality approach introduced by Schwartz in the 50's. A classical example is the extension of the Fourier transform to the space of tempered distributions. In this paper we extend the shearlet transform to distributions following the approach of Schwartz.

\par

Our work arises from the lack of a complete distributional framework for the shearlet transform in the literature and from the link between the shearlet transform with the Radon and the wavelet transforms, whose distribution theory is deeply investigated and well known subject in applied mathematics.  We refer respectively to \cite{FBhol1995, FBPRTV} and to \cite{FBhelgason99,FBhertle83} for the extension of the wavelet transform and the Radon transform to various generalized function spaces via a duality approach. The Lizorkin space plays a crucial role in the development of a distributional framework for these two classical transforms and it turns out to be a natural domain for the shearlet transform too. We recall that the Lizorkin space $\mathcal{S}_0 (\mathbb{R}^{d})$ consists of smooth and rapidly decreasing functions with vanishing moments of any order. Moreover, in \cite{FBkpsv2014} the authors show that the domain of the ridgelet transform can be enlarged to its dual space $\FBcS'_0(\FBR^2)$, known as the space of Lizorkin distributions. Their proofs widely exploit the intimate connection between the Radon, the ridgelet and the wavelet transforms, which also yields a relation formula between the shearlet transform and the ridgelet transform as we show in the Appendix in Proposition~\ref{FBprop:rideletwaveletradon}. This has in part inspired our work and we adapt several ideas of \cite{FBkpsv2014} to our context.
\par
Our main results are continuity theorems for the shearlet transform and its dual transform, called the shearlet synthesis operator, on various test function spaces. Precisely, we prove that the shearlet transform $\FBcS_\psi\colon \mathcal{S}_0(\FBR^2)\to \mathcal{S} (\mathbb{S}) $ and its dual transform $\FBcS_\psi^t\colon \mathcal{S}(\mathbb{S})\to \mathcal{S} (\FBR^2)$ are continuous, where $\mathcal{S}(\mathbb{S})$ is a certain space of highly localized functions (see Subsection \ref{FBsec:2} for the definition of $ \mathcal{S} (\mathbb{S}) $). Our continuity theorems hold for suitable choice of the admissible vector $\psi$ in the space $\mathcal{S}_0(\FBR^2)$ (see Section~3) and this is not a surprising condition. Indeed, as pointed out in wavelet analysis \cite{FBmallat09}, shearlet analysis \cite{FBgr11} and in the study of the Taylorlet transform \cite{FBfinkkahler19}, vanishing moments are crucial in order to measure the local regularity and to detect anisotropic structures of a signal. Then, we use these continuity results to extend the shearlet transform to the space of Lizorkin distributions following the approach in \cite{FBkpsv2014}. We show that the shearlet transform can be extended as a continuous map from $\FBcS'_0(\FBR^2)$ into $\FBcS'(\mathbb{S})$, where $\FBcS'(\mathbb{S})$ is the space of distributions of slow growth on $\FBR^2\times\FBR\times\FBR^{\times}$.  Observe that many important Schwartz distribution spaces, such as $\FBcE ' (\mathbb{R}^{d})$,
$\FBcO _C' (\mathbb{R}^{d})$, $ L^p  (\mathbb{R}^{d})$ and $\FBcD _{L^1} ' (\mathbb{R}^{d})$ are embedded into the
the space of Lizorkin distibutions  $\mathcal{S}_0 ' (\mathbb{R}^{d})$ (see e.g. \cite{FBkpsv2014}).
\par
When considering possible applications of our approach,
we notice that the rectified linear units (ReLUs), which are important examples of
unbounded activation functions in the context of deep learning neural networks, belong to the space of Lizorkin distributions, see
\cite{FBsomu17} for details.
\par
The chapter is organized as follows. In Section~2 we introduce the spaces that occur in our analysis. Then, we recall the definition and the basic properties of the wavelet transform and the Radon transform in polar and affine coordinates. Section~3
is devoted to an introduction of the shearlet transform and to recall one of the main results in \cite{FBbardemadeviodo}. Moreover, in Theorem~\ref{FBthm:continuityshearlettransform}, we give a sketch of the proof of the continuity of the shearlet transform $\FBcS_\psi$ on $\mathcal{S}_0(\FBR^2)$. In Section~4 we introduce and study the shearlet synthesis operator.  In particular, in Theorem \ref{FBthm:continuitysynthesisoperator} we prove its continuity on $ \mathcal{S} (\mathbb{S}) $.  The importance of this dual transform follows by the fact that it can be used to define the extension of the shearlet transform to the space of Lizorkin distibutions  $\mathcal{S}_0 '(\mathbb{R}^{2})$ in a natural way, as we show in Section \ref{FBsec:5}. We conclude our analysis with Theorem 4 which proves that our definition of the shearlet transform of distributions extends the ones considered so far, see e.g. \cite{FBkula12, FBgr11}, and it is consistent with those for test functions (see Definition \ref{FBdefsheartransf}). Moreover, Theorem 4 shows that our duality approach is equivalent to the one based on the coorbit space theory presented in \cite{dahlikeetal}.

\subsection{Notation}
\label{FBsubsec:1.1}
We briefly introduce the notation. We set $\mathbb{N} = \{ 0,1,2, \dots \}$, $ \mathbb{Z}_+$ denotes the set of positive integers,
$\FBR_+ =(0,+\infty)$, $\FBR^{\times}=\FBR\setminus\{0\}$ and $\mathbb{H}^{d + 1} = \mathbb{R}^d \times \mathbb{R}^{\times} $, $d \in \mathbb{Z}_+$. We also use the notation $ \mathbb{H}^{(m_1,\ldots,m_d,1)} $
$=\FBR^{m_1}\times\ldots\times\FBR^{m_d}\times\FBR^{\times}$, $ m_j \in \mathbb{Z}_+, $ $ j =1,\dots,d.$
When $ x, y \in \mathbb{R}^d $ and $ m \in
\mathbb{N}^d $, $|x|$ denotes the Euclidean norm, $x\cdot y$ their scalar product, $ \langle x \rangle = (1+|x|^2)^{1/2}$,
$ xy = x_1 y_1 + x_2 y_2 + \dots + x_d y_d, $ $ x^{m} = x_1^{m_1} \dots x_d^{m_d}$
 and
$\partial^{m}=\partial_x^{m} =
\partial_{x_1}^{m_1} \dots\partial_{x_d}^{m_d}$. We write also $\varphi^{(m)}=\partial^{m}\varphi$, $m\in\FBN^d$.
By a slight abuse of notation, the length of a multi-index
$ m \in \mathbb{N}^d $ is denoted by $ |m| = m_1 + \dots + m_d $ and
the meaning of $|\cdot| $ shall be clear from the context.
We write $ A\lesssim B $ when $ A \leq C \cdot B $ for
some positive constant $C$.
\par
For any $p\in[1,+\infty]$ we denote
by $L^p(\FBR^d)$ the Banach space of functions $f\colon\FBR^d\rightarrow\FBC$ that are $p$-integrable with respect to the Lebesgue measure $\D x$
and, if $p=2$, the corresponding scalar product and norm are
$\langle\cdot,\cdot\rangle$ and $\|\cdot\|$, respectively.
The Fourier transform is denoted by $\mathcal F$ both on
$L^2(\FBR^d)$ and on  $L^1(\FBR^d)$, where it is
defined by
\begin{equation*}
\mathcal F f({\xi}\,)= \int_{\FBR^d} f(x) {\rm e}^{-2\pi i\,
  {\xi}\cdot x } \D{x},\qquad f\in L^1(\FBR^d).
\end{equation*}
If $G$ is a locally compact group, we denote by $L^2(G)$ the Hilbert
space of square-integrable functions with respect to a left Haar
measure on $G$, and $C(G)$ denotes the space of continuous functions on $G$. If $A\in M_{d}(\FBR)$, the vector space of square $d\times d$ matrices with real entries, $^t\! A$ denotes its transpose and we denote the (real) general linear group of size $d\times d$ by ${\rm GL}(d,\FBR)$.
Finally, for every $b\in\FBR^d$, the translation operator acts on a
function $f:\FBR^d\to \FBC$ as
$
T_bf(x)=f(x-b)
$
and the dilation operator $D_a\colon L^p(\FBR^d)\to L^p(\FBR^d)$ is defined by $D_a f(x)=|a|^{-\frac{1}{2}}f(x/a)$ for every $a\in\FBR^{\times}$.
\par
The dual pairing between a test function space $ {\mathcal A}$ and its dual space of distributions
${\mathcal A'}$ is denoted by $(\: \cdot\:, \:\cdot\:)={_{\mathcal A'}(\: \cdot\:, \:\cdot \:)_{\mathcal A} }$
and we provide all distribution spaces with the strong dual topologies.
\par
The Schwartz space of rapidly decreasing smooth test functions is denoted by $\mathcal{S}(\mathbb{R}^{d})$
and $\mathcal{S}'(\mathbb{R}^{d})$ denotes its dual space of tempered distributions.  For the seminorms on $ \mathcal{S} (\FBR^d)$,
we make the choice
\begin{equation*}
\rho_\nu (\varphi ) = \sup_{x \in \FBR^d, |m|\leq \nu} \langle x \rangle ^\nu |\partial^{m} \varphi(x) |,
\end{equation*}  for every $ \nu \in \FBN$ and
$\varphi \in  \mathcal{S} (\FBR^d)$.
\section{Preliminaries}
In this section we first introduce the Lizorkin space of test functions, an important subspace of $\mathcal{S}(\mathbb{R}^{d})$ which plays a crucial role in our analysis. Afterwards, we recall the definition and the main properties of the wavelet transform and the Radon transform in polar and affine coordinates in order to recall in Section~\ref{FBsec:3} part of the results contained in \cite{FBbardemadeviodo}.
\subsection{The spaces}\label{FBsec:2}
In this subsection we introduce the spaces that occur in this chapter and we state some auxiliary results which we widely exploit in the proofs of our main results (see Lemma~\ref{FBLm:stability} and \ref{FBLm:mainlm} below).
In particular, the Lizorkin space $\mathcal{S}_0 (\mathbb{R}^{d})$ will play a crucial role in our analysis. It consists of rapidly decreasing functions with vanishing moments of any order. Precisely,
\begin{equation*}\mathcal{S}_0 (\mathbb{R}^{d}) = \left\{\varphi\in
\mathcal{S}(\mathbb{R}^{d}): \: \mu_m (\varphi) = 0,\ \forall m
\in \mathbb{N}^{d} \right\},
\end{equation*}
where $\mu_{m}(\varphi)=\int_{\mathbb{R}^{d}} x^{m}\varphi(x)dx$,
$m\in\mathbb{N}^{d}$.
The Lizorkin space  $\mathcal{S}_0 (\mathbb{R}^{d})$ is a closed subspace of
$\mathcal{S}(\mathbb{R}^{d})$ equipped with the relative
topology inhered from $\mathcal{S}(\mathbb{R}^{d})$ and
its dual space of Lizorkin distributions $\FBcS'_0(\FBR^d)$ is canonically isomorphic to the quotient of $\FBcS'(\FBR^d)$ by the space of polynomials
(cf. \cite{FBhol1995, FBkpsv2014}).
\par
We are also interested in the Fourier Lizorkin space $\hat{\mathcal{S}}_0 (\mathbb{R}^{d})$ which consists of rapidly decreasing functions that vanish in zero together with all their partial derivatives, i.e.
\begin{equation*}
\hat{\mathcal{S}}_0 (\mathbb{R}^{d})
= \left\{\varphi\in \mathcal{S}(\mathbb{R}^{d}): \: \partial ^{m} \varphi (0)= 0,\ \forall m \in \mathbb{N}^{d} \right\},
\end{equation*}
which is a closed subspace of $\mathcal{S}(\mathbb{R}^{d})$ too and we endow it with the relative
topology inhered from $\mathcal{S}(\mathbb{R}^{d})$.
\par
We observe that, since $\mathcal{S}_0 (\mathbb{R}^{d})$ and  $\hat{\mathcal{S}}_0 (\mathbb{R}^{d})$ are
closed subspaces of the nuclear space $ \mathcal{S} (\mathbb{R}^d) $, they are nuclear as well.
We denote by $X\hat{\otimes}Y$ the topological tensor product space obtained as the completion of $X\otimes Y$ in the inductive tensor product topology $\varepsilon$ or the projective tensor product topology $\pi$, see \cite{FBTreves1967} for details. Then, we have the following result.
\begin{lemma} \label{FBLm:stability}
The spaces $\mathcal{S}_0 (\mathbb{R}^{d})$ and $\hat{\mathcal{S}}_0 (\mathbb{R}^{d})$ are
closed under translations, dilations, differentiations and multiplications by  a polynomial.
Moreover, the Fourier transform is an isomorphism between $\mathcal{S}_0 (\mathbb{R}^{d})$ and
$\hat{\mathcal{S}}_0 (\mathbb{R}^{d})$ and we have the following canonical isomorphisms:
\begin{equation*}
\mathcal{S}_{0} (\mathbb{R}^{d}) \cong
\mathcal{S}_{0} (\mathbb{R}^{d_1}) \hat{\otimes} \mathcal{S}_{0} (\mathbb{R}^{d_2}),
%= \mathcal{S}_{0} (\mathbb{R}^{d_1}) \hat{\otimes}_\varepsilon \mathcal{S}_{0} (\mathbb{R}^{d_2}) =
%\mathcal{S}_{0} (\mathbb{R}^{d_1}) \hat{\otimes}_\pi \mathcal{S}_{0} (\mathbb{R}^{d_2}),
\end{equation*}
\begin{equation*}
\hat{\mathcal{S}}_{0} (\mathbb{R}^{d}) \cong
\hat{\mathcal{S}}_{0} (\mathbb{R}^{d_1}) \hat{\otimes} \hat{\mathcal{S}}_{0} (\mathbb{R}^{d_2}),
%= \hat{\mathcal{S}}_{0} (\mathbb{R}^{d_1}) \hat{\otimes}_\varepsilon \hat{\mathcal{S}}_{0} (\mathbb{R}^{d_2}) =
%\hat{\mathcal{S}}_{0} (\mathbb{R}^{d_1}) \hat{\otimes}_\pi \hat{\mathcal{S}}_{0} (\mathbb{R}^{d_2}),
\end{equation*}
where $d =d_1 + d_2 \in \mathbb{Z}_+$, and $\hat{\otimes} $ denotes the completion with respect to the $\varepsilon-$topology  or the $\pi$-topology.
\end{lemma}
\begin{proof}
The proof is based on classical arguments and we omit it (cf.  \cite[Theorem 51.6]{FBTreves1967} for the canonical isomorphisms).
\end{proof}
\par
The next Lemma is a reformulation of {\cite[Theorem 6.2]{FBmallat09}}.
\begin{lemma}[\cite{FBbartoluccithesis}]
Let $f\in   \mathcal{S}_{0} (\mathbb{R}^d) $. Then, for any given $m\in\FBN^d$ there exists $g\in  \mathcal{S}_{0} (\mathbb{R}^d)$ such that
\begin{equation*}
\FBcF f (\xi) = \xi ^m \FBcF g (\xi),\qquad \xi \in \mathbb{R}^d,
\end{equation*}
and vice versa.\label{FBLm:mainlm}
\end{lemma}
\begin{proof}
We start proving the above statement for $d=1$ and $m=1$. Let $f\in   \mathcal{S}_{0} (\mathbb{R}) $ and consider
\begin{equation*}
g(x)=\int_{-\infty}^xf(t)\D t=-\int_{x}^{+\infty}f(t)\D t,
\end{equation*}
where in the second equality we  use the fact that $f\in \mathcal{S}_{0} (\mathbb{R})$. For every $k\in\FBN$ and $x>0$
\begin{align*}
\langle x \rangle^{k} |g(x)|&=|\int_{x}^{+\infty}(1+x^2)^{\frac{k}{2}}f(t)\D t|\leq\int_{x}^{+\infty}(1+t^2)^{\frac{k}{2}}|f(t)|\D t\\
&\leq\rho_{2k+4}(f)\int_{-\infty}^{+\infty}(1+t^2)^{\frac{k}{2}}\frac{1}{(1+t^2)^{k+2}}\D t<+\infty.
\end{align*}
Analogously, for $x<0$ it holds
\begin{align*}
\langle x \rangle^{k} |g(x)|&=|\int_{-\infty}^{x}(1+x^2)^{\frac{k}{2}}f(t)\D t|\leq\int_{-\infty}^{x}(1+t^2)^{\frac{k}{2}}|f(t)|\D t\\
&\leq\rho_{2k+4}(f)\int_{-\infty}^{+\infty}(1+t^2)^{\frac{k}{2}}\frac{1}{(1+t^2)^{k+2}}\D t<+\infty.
\end{align*}
Thus, $g$ is a well defined function and $\sup_{x\in\FBR}\langle x \rangle^{k} |g(x)|<+\infty$ for every $k\in\FBN$.
Moreover, $ g'(x) = f(x),$ so that $\sup_{x\in\FBR}\langle x \rangle^{k} |g^{(l) }(x)|<+\infty$ for every $k\in\FBN$
and $l\geq 1$.
Therefore, $ g \in   \mathcal{S} (\mathbb{R}) $.
Furthermore, for any $n\in\FBN$, we have that
\begin{equation*}
\int_{-\infty}^{+\infty}x^ng(x)\D x=-\int_{-\infty}^{+\infty}x^{n+1}g'(x)\D x=-\int_{-\infty}^{+\infty}x^{n+1}f(x)\D x=0.
\end{equation*}
Hence, $g\in\FBcS_0(\FBR)$ and by the definition of $g$ we have
$$
\FBcF f(\xi)=\FBcF g'(\xi)=(2\pi i) \xi \FBcF g(\xi),\qquad\xi\in\FBR.
$$
The opposite direction is obviously true since the space $\FBcS_0(\FBR)$ is closed under multiplication by a polynomial and this concludes the proof for $d=1$ and $m=1$. The analogous statement holds true for $m>1$ by iterating the above proof $m$-times.
The case $d>1$ follows by analogous computations.
\end{proof}
\par
By Lemma~\ref{FBLm:mainlm}, if $f\in   \mathcal{S}_{0} (\mathbb{R}^2) $, then for any given $k,l \in \FBN $ there exists $g \in   \mathcal{S}_{0} (\mathbb{R}^2) $ such that
\begin{equation*}
\FBcF f (\xi_1, \xi_2) = \xi_1 ^k \xi_2 ^l \FBcF g (\xi_1, \xi_2), \;\;\; (\xi_1, \xi_2) \in \mathbb{R}^2.
\end{equation*}
Moreover, it is worth observing that by Lemmas \ref{FBLm:stability} and \ref{FBLm:mainlm},   $f\in   \mathcal{S}_{0} (\mathbb{R}^d) $ if and only if it satisfies the directional vanishing moments:
$$
\int_{\mathbb{R}} x_j ^{m} f(x_1, x_2, \dots, x_d) dx_j = 0, \forall m \in \mathbb{N}, \;\;\; j = 1,\dots,d.
$$
\par
The space $ \mathcal{S} (\mathbb {H}^{(d,d-1,1)}) $ of highly localized
functions (see also \cite{FBhol1995}) consists of  the functions $ \Phi
\in C^{\infty} (\mathbb{H}^{(d,d-1,1)}) $ such that the seminorms
\begin{align*}
&\rho_{k_1,k_2,l,m} ^{\alpha_1,\alpha_2, \beta,\gamma}(\Phi)=\\
&\sup_{(b,s,a)\in
\mathbb {H}^{(d,d-1,1)}}\langle b_1 \rangle ^{k_1} \,\langle \tilde {b} \rangle ^{k_2} \,\langle s \rangle ^l \,
\left(|a|^{m}+\frac
{1}{|a|^{m}}\right)\left| \partial^{\gamma} _{a}  \partial^{\beta} _{s}
 \partial^{\alpha_2} _{\tilde b}
\partial^{\alpha_1} _{b_1}
\Phi (b,s,a)\right |
\end{align*}
are finite for all $k_1,m,\alpha_1,\gamma \in \mathbb{N}$, $k_2,l,\alpha_2, \beta \in \mathbb{N}^{d-1}$ and where $b=(b_1,\tilde b)\in\FBR\times\FBR^{d-1}$.
In particular, when $d=2$, we denote $\mathbb{S}:=\mathbb{H}^{(d,d-1,1)} = \FBR^2\times\FBR\times\FBR^{\times}$ and $ \mathcal{S} (\mathbb{S}) $ consists of the functions
$ \Phi\in C^{\infty} (\mathbb{S}) $ such that the seminorms
\begin{align} \label{FBnorma-hl2}
\nonumber&p_{k_1,k_2,l,m}^{\alpha_1,\alpha_2, \beta,\gamma}(\Phi) \\
&=\sup_{((b_1,b_2),s,a)\in\mathbb{S}}
\langle b_1 \rangle ^{k_1} \,\langle b_2 \rangle ^{k_2} \,\langle s \rangle ^l \, \left(|a|^{m}+\frac {1}{|a|^{m}}\right)
\left|  \partial^{\gamma} _{a}  \partial^{\beta} _{s} \partial^{\alpha_2} _{b_2} \partial^{\alpha_1} _{b_1} \Phi ((b_1,b_2),s,a)\right |
\end{align}
are finite for all $k_1,k_2,l,m,\alpha_1,\alpha_2,\beta,\gamma \in \mathbb{N}$.
The topology of  $ \mathcal{S} (\mathbb{S}) $  is defined by means of the seminorms \eqref{FBnorma-hl2}. Its dual $ \mathcal{S}' (\mathbb{S}) $ will play a crucial role in
the definition of the shearlet transform of Lizorkin distributions
since it contains the range of this transform.
We fix ${\rm d}\mu(b,s,a)=|a|^{-3}\D b\D s\D a$ as the standard measure on $\mathbb{\mathbb{S}}$, where ${\rm d}b$, ${\rm d}s$ and ${\rm d}a$ are the Lebesgue measures on $\FBR^2$, $\FBR$ and $\FBR^{\times}$, respectively.
If $F$ is a function of at most polynomial growth on $\mathbb{S}$, i.e., if there exist
$C,\nu_1,\nu_2,\nu_3>0$ such that
\begin{equation*}
|F(b,s,a)|\leq C \langle b \rangle^{\nu_1} \langle s \rangle^{\nu_2} \left(|a|^{\nu_3}+\frac
{1}{|a|^{\nu_3}}\right),\qquad (b,s,a)\in \mathbb{S},
\end{equation*}
then we identify $F$ with an element of $ \mathcal{S}' (\mathbb{\mathbb{S}}) $ by means of the equality
\begin{equation}\label{FBdualitysynthesisspace}
( F,\Phi ) = \int_{\FBR^{\times}}\int_{\FBR}\int_{\FBR^2}F(b,s,a)\Phi(b,s,a)\frac{\D b\D s\D a}{|a|^3},
\end{equation}
for every $ \Phi\in  \mathcal{S} (\mathbb{\mathbb{S}})$.
\subsection{The wavelet transform}
The one-dimensional affine group $\mathbb{W}$ is the semidirect product $\FBR\rtimes\FBR^{\times}$ with group operation
\begin{equation*}
(b,a)(b',a')=(b+ab',aa')
\end{equation*}
and left Haar measure $|a|^{-2}\D b\D a$. It acts on $L^2(\FBR)$ by means of the square-integrable representation
\begin{equation*}
W_{b,a}f(x)=|a|^{-\frac{1}{2}}f\left(\frac{x-b}{a}\right),
\end{equation*}
or, equivalently, in the frequency domain
\begin{equation} \label{FBwavefreq}
\FBcF W_{b,a}f(\xi)=|a|^\frac{1}{2} e^{-2\pi i b \xi}\FBcF f(a\xi).
\end{equation}
The wavelet transform is then $\FBcW_{\psi}f(b,a)=\langle f, W_{b,a}\psi\rangle$, which is a multiple of an isometry from $L^2(\FBR)$ into $L^2(\mathbb{W}, |a|^{-2}\D b\D a)$ provided that $\psi\in L^2(\FBR)$ satisfies the admissibility condition, namely the Calder\'on equation,
\begin{equation}\label{FBeqn:calderon}
0<\int_{\FBR}\frac{|\FBcF\psi(\xi)|^2}{|\xi|}\D\xi<+\infty
\end{equation}
and, in such a case, $\psi$ is called a one-dimensional wavelet.
We refer to \cite{FBhol1995, FBPRTV} for the extension of the wavelet transform to Lizorkin distributions.
\subsection{The affine Radon transform}\label{FBsec:radontransform}
The Radon transform of a signal $f$ is a function on the affine projective space
$\mathbb P^1\times \FBR=\{\Gamma\mid \Gamma \text{ line of }\FBR^2\}$ whose value at a line is the integral of $f$ along that line. It is usually defined by parametrizing the lines by pairs $(\theta,q)\in[-\pi,\pi)\times\FBR$ as
\begin{equation*} \Gamma_{\theta,q}= \{ (x,y)\in \FBR^2\mid \cos{\theta}x+ \sin{\theta}y=q\} \end{equation*}
and we refer to it as the polar Radon transform, see \cite{FBhelgason99}. Precisely, for every $f\in L^1(\FBR^2)$, the polar Radon transform of $f$ is the map $\FBcR^{\rm pol}f\colon [-\pi,\pi)\times\FBR\to\FBC$ defined by
\begin{equation}\label{FBeqn:Radonpolare}
{\FBcR}^{\rm pol}f(\theta,q)
=\int_{\FBR} f(q\cos \theta -y\sin \theta, q\sin \theta + y \cos \theta)\ {\rm d}y,
\end{equation}
where the equality \eqref{FBeqn:Radonpolare} holds for almost every $(\theta,q)\in[-\pi,\pi)\times\FBR$.
\par
We can also label the normal vector to a line by
affine coordinates, that is
\begin{equation*} \Gamma_{v,t}= \{ (x,y)\in \FBR^2\mid x+ v y=t\}, \end{equation*}
where the correspondence is $v=\tan\theta$ and $t=q/\cos\theta$.
Horizontal lines can not be represented by this parametrization, but they constitute a negligible
set with respect to the natural measure on $\mathbb P^1\times \FBR$.
The affine Radon transform of any $f\in L^1(\FBR^2)$ is the function $\mathcal{R}^{\rm aff} f:\FBR^2\to \FBC$ defined by
\begin{equation*}
\mathcal{R}^{\rm aff} f (v,t)=  \int_{\FBR}f(t-v y,y) \,\D y, \qquad\text{a.e.}\ (v,t)\in\FBR^2,
\end{equation*}
and it is related to the polar Radon transform by
\begin{equation}
 \mathcal{R}^{\rm aff} f(v,t)=\frac{1}{\sqrt{1+v^2}} \mathcal{R}^{\rm pol} f(\arctan v,\frac{t}{\sqrt{1+v^2}}).
\label{FBeqn:polaff}
\end{equation}
We refer to \cite{FBbardemadeviodo} for the proof.
\par
The choice of the affine parametrization is particularly
well-adapted to the mathematical structure of the shearlet transform (see also \cite{FBgr11} and \cite{FBbardemadeviodo}).
\par
The next result is a formulation of the Fourier slice theorem written for the affine Radon transform. The function $f$ to which $ \mathcal{R}^{\rm aff}$ is applied is taken in $L^1(\FBR^2)\cap L^2(\FBR^2)$.
\begin{proposition}[{\cite[Proposition 6]{FBbardemadeviodo}}]
Define $\psi:\FBR\times(\FBR\setminus\{0\})\rightarrow\FBR^2$ by  $\psi(v,\tau)=(\tau,\tau v)$. For every  $f\in L^1(\FBR^2)\cap L^2(\FBR^2)$ there exists a negligible set $E\subseteq\FBR$ such that for all $v\not\in E$ the function $\mathcal{R}^{\rm aff}f(v,\cdot)$ is in $L^2(\mathbb{R})$ and satisfies
\begin{equation}
\label{FBeqn:fst2}
\mathcal{R}^{\rm aff}f(v,\cdot)=\mathcal{F}^{-1}[\mathcal{F}f\circ\psi(v,\cdot)].
\end{equation}
\label{FBprop:fstaff}
\end{proposition}
We refer to \cite{FBbardemadeviodo} for the proof of Proposition~\ref{FBprop:fstaff} and we refer to \cite{FBhelgason99} as a classical reference for the Radon transform theory.

\section{The Shearlet transform}
\label{FBsec:3}
In this section we introduce the shearlet transform. Then, we recall the intertwining result proved in \cite{FBbardemadeviodo}, see Theorem~\ref{FBteo:teocentrale} below, and we give an idea of how it
can be exploited to derive continuity properties of the shearlet transform when restricted to the Lizorkin space of test functions, see Theorem~\ref{FBthm:continuityshearlettransform}. We refer to \cite{FBfolland16}
 as a classical reference for the theory of group representations of locally compact groups and to \cite{FBkula12} for a complete overview of the shearlet analysis.
\par
The standard shearlet group is the semidirect product
$G=\FBR^2\rtimes K$, where $K=\{S_sA_a\in {\rm GL}(2,\FBR):s\in\FBR,a\in\FBR^{\times}\}$ with
\begin{equation*}
S_s=\left[\begin{matrix}1 & -s\\ 0 & 1\end{matrix}\right],\qquad A_a=a\left[\begin{matrix}1 & 0\\ 0 & |a|^{-1/2}\end{matrix}\right].
\end{equation*}
We can identify the element $S_sA_a$ with the pair $(s,a)$ and write $(b,s,a)$ for the elements in $G$. With this identification the product law amounts to
\begin{equation*}
(b,s,a)(b',s',a')=(b+S_sA_ab',s+|a|^{1/2}s',aa').
\end{equation*}
A left Haar measure of $G$ is
\begin{equation*}
{\rm d}\mu(b,s,a)=|a|^{-3}{\rm d}b{\rm d}s{\rm d}a,
\end{equation*}
with ${\rm d}b$, ${\rm d}s$ and ${\rm d}a$ the Lebesgue measures on $\FBR^2$, $\FBR$ and $\FBR^{\times}$, respectively.
The group $G$ acts on $L^2(\FBR^2)$ via the square-integrable representation
\begin{equation*}
S_{b,s,a}f(x)=|a|^{-3/4}f(A_a^{-1}S_s^{-1}(x-b)),
\end{equation*}
or, equivalently, in the frequency domain
\begin{equation} \label{FBshearfreq}
\FBcF S_{b,s,a}f(\xi)=|a|^{3/4} e^{-2\pi i b \xi}\FBcF f(A_a{^t\!S_s}\xi).
\end{equation}
\begin{definition} \label{FBdefsheartransf}
Fix $\psi\in L^2(\FBR^2)$. The shearlet transform associated to $\psi$ is the map $\FBcS_{\psi}\colon L^2(\FBR^2)\to C(G)\cap L^{\infty}(G,{\rm d}\mu)$ defined by
\begin{equation*}
\FBcS_{\psi}f(b,s,a)=\langle f,S_{b,s,a}\psi\rangle
= |a|^{-\frac{3}{4}}\int_{\FBR^2}f(x)\overline{\psi(A_{a}^{-1} S_{s}^{-1}(x-b))}\D x.
\end{equation*}
\end{definition}
\par
It is well-known that the shearlet transform $\FBcS_{\psi}$
is a non-trivial multiple of an isometry from $L^2(\FBR^2)$ into $L^2(G)$ provided that $\psi\in L^2(\FBR^2)$ satisfies the admissibility condition
\begin{equation} \label{FBeqn:admvect}
0<C_{\psi}=\int_{\FBR^2}\frac{|\FBcF\psi(\xi)|^2}{|\xi_1|^2}\D\xi<+\infty,
\end{equation}
where $\xi=(\xi_1,\xi_2)\in\FBR^2$, or equivalently
\begin{equation*}
\int_{\FBR^{\times}}\int_{\FBR}|\FBcF\psi(A_a{^{t}\!S_s}\xi)|^2\D s\frac{\D a}{|a|^{3/2}}=C_{\psi},\qquad \text{for a.e. $\xi\in\FBR^2/\{0\}$},
\end{equation*}
see e.g. \cite{FBdahlke2008}. Furthermore, in such case, we have the reconstruction formula
\begin{equation}\label{FBreconstructionformulashearlet}
 f = \frac{1}{C_{\psi}}  \int_{\FBR^{\times}}\int_{\FBR}\int_{\FBR^2}
 \FBcS_{\psi}f(b,s,a) \,     S_{b,s,a}\psi\ \frac{\D b\D s\D a}{|a|^3},
\end{equation}
where the integral converges in the weak sense.
%From now on, when we consider an admissible vector $\psi$, we suppose $C_{\psi}=1$.
\par
We now recall part of the results in \cite{FBbardemadeviodo}. We fix $\psi\in L^2(\FBR^2)$ of the form
\begin{equation}\label{FBeqn:shearletfactorization2}
\FBcF\psi(\xi_1,\xi_2)=\FBcF\psi_1(\xi_1)\FBcF\psi_2\left(\frac{\xi_2}{\xi_1}\right),\quad (\xi_1,\xi_2)\in\FBR^2,\, \xi_1\ne0,
\end{equation}
with $\psi_1\in L^2(\FBR)$ satisfying the conditions
\begin{equation}\label{FBeq:4}
0<\int_{\FBR}\frac{|\FBcF\psi_1(\tau)|^2}{|\tau|}\ {\rm d}\tau<+\infty,\qquad\int_{\FBR}|\tau|^2|\FBcF\psi_1(\tau)|^2\ {\rm d}\tau<+\infty
\end{equation}
and $\psi_2\in L^2(\FBR)$.
Then, $\psi$ satisfies the admissible condition \eqref{FBeqn:admvect} and the function $\chi_1\in L^2(\FBR)$ defined by
\begin{equation}\label{FBeqn:phi}
\mathcal{F}\chi_1(\tau)=|\tau|\mathcal{F}\psi_{1}(\tau)
\end{equation}
is a one-dimensional wavelet, i.e. it satisfies \eqref{FBeqn:calderon}. We are now ready to state one of the central results in \cite{FBbardemadeviodo} which shows that the shearlet transform is the composition of the affine Radon transform with a one-dimensional wavelet, followed by a convolution with a scale-dependent filter.
\begin{theorem}[{\cite[Corollary 12]{FBbardemadeviodo}}]
For any $f\in L^1(\FBR^2)\cap L^2(\FBR^2)$ and $((b_1,b_2),s,a)\in \FBR^2\times\FBR\times\FBR^{\times}$,
\begin{equation}\label{FBeqn:achamain}
\mathcal{S}_{\psi}f((b_1,b_2),s,a)=|a|^{-\frac{3}{4}}\int_{\mathbb{R}}\mathcal{W}_{\chi_1}(\FBcR^{\rm aff} f(v,\cdot))(b_1+vb_2,a)\overline{\phi_2\left(\frac{v-s}{|a|^{1/2}}\right)}\ {\rm d}v,
\end{equation}
where $\phi_2=\mathcal{F}\psi_2$.\label{FBteo:teocentrale}
\end{theorem}
In \eqref{FBeqn:achamain} the wavelet transform is one-dimensional and acts on the variable $t$. We refer to \cite{FBbardemadeviodo} for the proof.
\par
By direct computation, applying the Plancherel theorem, the Fourier slice theorem \eqref{FBeqn:fst2} and equation \eqref{FBwavefreq}, we obtain the useful formula
\begin{align}\label{FBeq:shearcoefffrequency}
&\nonumber\mathcal{S}_{\psi}f((b_1,b_2),s,a)=|a|^{-\frac{3}{4}}\int_{\mathbb{R}}\mathcal{W}_{\chi_1}(\mathcal{R}^{\rm aff}f(v,\cdot))(b_1+vb_2,a)\overline{\phi_2\left(\frac{v-s}{|a|^\frac{1}{2}}\right)}{\rm d}v\\
\nonumber&=|a|^{-\frac{1}{4}}\int_{\mathbb{R}}\int_{\FBR}\FBcF\mathcal{R}^{\rm aff}f(v,\cdot)(\tau)\overline{\FBcF\chi_1(a\tau)}e^{2\pi i \tau(b_1+vb_2)}\D\tau\overline{\phi_2\left(\frac{v-s}{|a|^\frac{1}{2}}\right)}{\rm d}v\\
&=|a|^{-\frac{1}{4}}\int_{\mathbb{R}}\int_{\FBR}\FBcF f(\tau,\tau v)\overline{\FBcF\chi_1(a\tau)}e^{2\pi i \tau(b_1+vb_2)}\overline{\phi_2\left(\frac{v-s}{|a|^\frac{1}{2}}\right)}\D\tau{\rm d}v,
\end{align}
for every $((b_1,b_2),s,a)\in\FBR^2\times\FBR\times\FBR^{\times}$. Alternatively, we can obtain formula~\eqref{FBeq:shearcoefffrequency} as a direct consequence of equations~\eqref{FBshearfreq} and \eqref{FBeqn:shearletfactorization2} combined with the Plancherel theorem and a change of variable in affine coordinates
\begin{align}
&\nonumber\mathcal S_\psi f(b,s,a)=\langle \FBcF f,\FBcF S_{b,s,a},\psi\rangle\\
&\nonumber=|a|^{\frac{3}{4}}\int_{\FBR}\int_{\FBR}\mathcal{F} f(\xi_1,\xi_2)e^{2\pi i (b_1\xi_1+b_2\xi_2)}\overline{\mathcal{F}\psi(a\xi_1,a|a|^{-\frac{1}{2}}(\xi_2-s\xi_1))}{\rm d}\xi_1{\rm d}\xi_2\\
&\nonumber=|a|^{\frac{3}{4}}\int_{\FBR}\int_{\FBR}\mathcal{F} f(\tau,\tau v)e^{2\pi i \tau(b_1+vb_2)}\overline{\mathcal{F}\psi(a\tau,a|a|^{-\frac{1}{2}}\tau(v-s))}|\tau|{\rm d}\tau{\rm d}v\\
&\nonumber=|a|^{\frac{3}{4}}\int_{\FBR}\int_{\FBR}\mathcal{F} f(\tau,\tau v)e^{2\pi i \tau(b_1+vb_2)}\overline{|\tau|\mathcal{F}\psi_1(a\tau)\mathcal{F}\psi_2\left(\frac{v-s}{|a|^{\frac{1}{2}}}\right)}{\rm d}\tau{\rm d}v,
 \end{align}
 which is exactly formula~\eqref{FBeq:shearcoefffrequency} since by definition $\mathcal{F}\chi_1=|\cdot|\mathcal{F}\psi_{1}$ and $\phi_2=\mathcal{F}\psi_2$.
\par
We are now ready to state our first main result. From now on, everytime we consider an admissible vector $\psi$, we assume that it is of the form \eqref{FBeqn:shearletfactorization2} with $\chi_1\in\mathcal{S}_0(\FBR)$ and $\phi_2\in \FBcS(\FBR)$. We observe that if $\chi_1\in\mathcal{S}_0(\FBR)$, then $\psi_1$ defined by \eqref{FBeqn:phi} as $\mathcal{F}\psi_1(\tau)=|\tau|^{-1}\mathcal{F}\chi_{1}(\tau)$ satisfies \eqref{FBeq:4}. Furthermore, under these assumptions, $\FBcF\psi$ extends to a function belonging to the Fourier Lizorkin space $\hat{\mathcal{S}}_0 (\mathbb{R}^{2})$ and, with slight abuse of notation, $\psi$ denotes both the admissible vector defined by \eqref{FBeqn:shearletfactorization2} and its Schwartz extension over $\FBR^2$.
%%%%%%
%
%\textcolor{gray}{Finally, for simplicity, we restrict ourselves to the connected version of the shearlet group, which corresponds to restricting %the scale parameter $a$ over $\FBR^{\times}$.}
%
%%%%%
\begin{theorem}[\cite{FBbartoluccithesis}]
The shearlet transform  $ \mathcal S_\psi$ is a continuous mapping from
$ \mathcal{S}_{0} (\mathbb{R}^2)$ into $ \mathcal{S}(\mathbb {S}) $.\label{FBthm:continuityshearlettransform}
\end{theorem}
\begin{proof}
%Recall, $\rho_\nu (\varphi ) = \sup_{x \in \FBR^d, |m|\leq \nu} \langle x \rangle ^\nu |\partial ^{m} \varphi (x) |, $ $ \nu \in \N ,$ $\varphi \in  \mathcal{S} (\FBR)$.
We give a sketch of the proof and we refer to \cite{FBbartoluccithesis} for the details. The goal is to prove that for every $f \in \mathcal{S}_{0} (\mathbb{R}^2)$, given $k_1,k_2,l,m,\alpha_1,\alpha_2,\beta,\gamma \in \mathbb{N}$, there exist $ \nu \in \FBN$ such that
\begin{equation*}
\rho_{k_1,k_2,l,m} ^{\alpha_1,\alpha_2, \beta,\gamma} ( \mathcal S_\psi f) \lesssim  \rho_{\nu} (f ).
\end{equation*}
The first step in the proof is to show that, without loss of generality, we can assume $\alpha_1=\alpha_2=\beta=\gamma=k_1=k_2=l=0$. Then, it is enough to prove that for every $m\in\FBN$
\begin{equation*}
\rho_{0,0,0,m} ^{0,0,0,0} ( \mathcal S_\psi f) \lesssim  \rho_{\nu} (f ),
\end{equation*}
for some $ \nu \in \FBN$. Our approach mimics the one followed in \cite{FBkpsv2014} to prove the continuity of the ridgelet transform on the Lizorkin space of test functions. Formula \eqref{FBeq:shearcoefffrequency} and Lemma~\ref{FBLm:mainlm} play a crucial role throughout all the proof. In particular, equation \eqref{FBeq:shearcoefffrequency} has the same role of formula (9) in \cite{FBkpsv2014}.
\end{proof}
\section{The shearlet synthesis operator}%%%%%%%%%%%%%%%%%%%%%%%%%%%%%%THESHEARLESYNTHESISOPERATOR
\label{FBsec:4}
In this section we introduce the shearlet synthesis operator and we prove its continuity on the space $\mathcal{S}(\mathbb{S})$ of highly localized  functions, see Theorem~\ref{FBthm:continuitysynthesisoperator} below.
\par
The reconstruction formula \eqref{FBreconstructionformulashearlet}  suggests to define a linear operator $\FBcS_{\psi}^t$ which maps functions over $\mathbb{S} = \FBR^2\times\FBR\times\FBR^{\times}$ to functions over the Euclidean plane $\FBR^2$. Given $\psi\in \FBcS(\FBR^2)$, we
define the shearlet synthesis operator $\FBcS_{\psi}^t$ by
\begin{align}\label{FBsynthesisoperator}
&\FBcS_{\psi}^t \Phi(x)=\int_{\FBR^{\times}}\int_{\FBR}\int_{\FBR^2} \Phi(b,s,a)\,\, {S}_{b,s,a}\psi(x)
\,\,\frac{\D b \D s  \D a}{|a|^3},\qquad x\in\FBR^2,
\end{align}
for any function $\Phi$ for which the integral converges. For example,
the integral in \eqref{FBsynthesisoperator} is absolutely convergent if $\Phi\in \FBcS(\mathbb{S})$.
Furthermore, if $f\in L^1(\FBR^2)\cap L^2(\FBR^2)$ and $\Phi\in \FBcS(\mathbb{S})$, then by Fubini theorem we have that
\begin{align}\label{FBeq:dualityrelationfirst}
\nonumber\int_{\FBR^2}f(x)\FBcS_{\overline{\psi}}^t\Phi(x)\D x&=\int_{\FBR^2}f(x)\int_{\FBR^{\times}}\int_{\FBR}\int_{\FBR^2} \Phi(b,s,a)
\,\, \overline{{S}_{b,s,a}\psi(x)}\,\,\frac{\D b \D s  \D a}{|a|^3}\D x\\
\nonumber&=\int_{\FBR^{\times}}\int_{\FBR}\int_{\FBR^2}\Phi(b,s,a)\int_{\FBR^2}f(x)\overline{{S}_{b,s,a}\psi(x)}\D x
\,\,\frac{\D b \D s  \D a}{|a|^3}\\
&=\int_{\FBR^{\times}}\int_{\FBR}\int_{\FBR^2}\FBcS_{\psi}f(b,s,a)\Phi(b,s,a)\,\,\frac{\D b \D s  \D a}{|a|^3}.
\end{align}
In Theorem~\ref{FBthm:continuitysynthesisoperator} we will show that the shearlet synthesis operator $\mathcal S_\psi^t$ is a continuous operator from $\mathcal{S}(\mathbb{S})$ into $\mathcal{S}(\mathbb{R}^2)$. Then, since $L^1(\FBR^2)\cap L^2(\FBR^2)$ naturally embeds into $\FBcS'(\FBR^2)$, by the identification \eqref{FBdualitysynthesisspace}, we may write \eqref{FBeq:dualityrelationfirst} as
\begin{equation*}
( f, \FBcS_{\overline{\psi}}^t\Phi ) = (\FBcS_{\psi}f,\Phi ).
\end{equation*}
This duality relation will motivate our definition of the distributional shearlet transform in Section \ref{FBsec:5}.
\par
We recall that we consider admissible vectors $\psi$ of the form \eqref{FBeqn:shearletfactorization2} with $\chi_1$ defined by \eqref{FBeqn:phi} in $\mathcal{S}_0(\FBR)$ and $\phi_2=\mathcal{F}\psi_2\in \FBcS(\FBR)$.

\begin{theorem}[\cite{FBbartoluccithesis}]
The shearlet synthesis operator $\mathcal S_\psi^t$ is a bounded operator from $\mathcal{S}(\mathbb{S})$ into $\mathcal{S}_0 (\mathbb{R}^2)$. \label{FBthm:continuitysynthesisoperator}
\end{theorem}

\begin{proof}
We start proving the continuity.
We need to show that for every $\Phi\in\FBcS(\mathbb{S})$ and $\nu \in \mathbb{N}$, there exist $k_1,k_2,l,m,\alpha_1,\alpha_2,\beta,\gamma\in \FBN$ and a positive constant $C$ such that
\begin{equation*}
\rho_{\nu}(S^t_{\psi}\Phi)\leq C \rho_{k_1,k_2,l,m} ^{\alpha_1,\alpha_2, \beta,\gamma} (\Phi).
\end{equation*}
We will use the fact that the families $\hat{\rho}_{\nu}(\chi)=\rho_{\nu}(\FBcF\chi)$
and $\hat\rho_{k_1,k_2,l,m} ^{\alpha_1,\alpha_2, \beta,\gamma}(\Phi)=\rho_{k_1,k_2,l,m} ^{\alpha_1,\alpha_2, \beta,\gamma}(\FBcF\Phi)$,
where $\FBcF\Phi$ denotes the Fourier transform of $\Phi$ with respect to the variable $b$, are bases of seminorms for the topologies of $\FBcS_0(\FBR)$ and $\mathcal{S}(\mathbb {S})$, respectively (cf. \cite{FBkpsv2014}). Furthermore, by Plancherel theorem, Fubini theorem, equation \eqref{FBshearfreq} and by the expression of the admissible vector, we obtain
the following formula for the shearlet synthesis operator
\begin{align}\label{FBfrequencysynthesisoperator}
\nonumber&\FBcS_{\psi}^t \Phi(x)\\
\nonumber&=\int_{\FBR^{\times}}\int_{\FBR} |a|^{\frac{3}{4}}\int_{\FBR^2} \FBcF \Phi(\mathbf{\xi},s,a)\,\, e^{2\pi i x\cdot\xi} \overline{\FBcF \psi(-a\xi_1,a|a|^{-\frac{1}{2}}(-\xi_2+s\xi_1))}
\,\,\D \xi_1\D\xi_2\frac{ \D s \D a }{|a|^3}\\
\nonumber&=\int_{\FBR^2}e^{2\pi i x\cdot\xi} \int_{\FBR^{\times}}\int_{\FBR} |a|^{\frac{3}{4}} \FBcF \Phi(\mathbf{\xi},s,a)\,\, \overline{\FBcF \psi(-a\xi_1,a|a|^{-\frac{1}{2}}(-\xi_2+s\xi_1))}
\,\,\frac{ \D s \D a }{|a|^3}\D \xi_1\D \xi_2\\
\nonumber&=\int_{\FBR^2}e^{2\pi i x\cdot\xi} \int_{\FBR^{\times}}\int_{\FBR} |a|^{\frac{3}{4}} \FBcF \Phi(\mathbf{\xi},s,a)\,\, \overline{\FBcF \psi_1(-a\xi_1)\FBcF\psi_2\left(\frac{\xi_2/\xi_1-s}{|a|^{\frac{1}{2}}}\right)}
\,\,\frac{ \D s \D a }{|a|^3}\D \xi_1\D \xi_2\\
&=\int_{\FBR^2}e^{2\pi i x\cdot\xi} \int_{\FBR^{\times}}\int_{\FBR} |a|^{-\frac{1}{4}} \FBcF \Phi(\mathbf{\xi},s,a)\,\,\overline{\frac{\FBcF \chi_1(-a\xi_1)}{ |\xi_1|}\phi_2\left(\frac{\xi_2/\xi_1-s}{|a|^{\frac{1}{2}}}\right)}
\,\,\frac{ \D s \D a }{|a|^3}\D \xi_1\D \xi_2.
\end{align}
\par
By formula \eqref{FBfrequencysynthesisoperator}, for any given $\alpha\in  \mathbb{Z}_+$ we have that
\begin{align*}
&|\partial_{x_1}^{\alpha}(S^t_{\psi}\Phi)(x_1,x_2)|\\
&\lesssim\int_{\FBR^2}\int_{\FBR^{\times}}\int_{\FBR} |a|^{-\frac{1}{4}} |\FBcF \Phi(\mathbf{\xi},s,a)|\,\,\frac{|\FBcF \chi_1(-a\xi_1)|}{|\xi_1|^{1-\alpha}}\left|\phi_2\left(\frac{\xi_2/\xi_1-s}{|a|^{1/2}}\right)\right|
\frac{ \D s \D a }{|a|^3}\D \xi_1\D \xi_2\\
&\lesssim\int_{\FBR^2}\int_{|a|<\epsilon}\int_{\FBR} \frac{(1+|\xi|^2)^{N/2}}{(1+|\xi|^2)^{N/2}}(1+s^2)|a|^{-\frac{13}{4}} |\xi_1|^{\alpha-1}|\FBcF \Phi(\mathbf{\xi},s,a)|\,\,|\FBcF \chi_1(-a\xi_1)|\\
&\left|\phi_2\left(\frac{\xi_2/\xi_1-s}{|a|^{1/2}}\right)\right|
\frac{ \D s \D a }{(1+s^2)}\D \xi_1\D \xi_2\\
&+ \int_{\FBR^2}\int_{|a|>\epsilon}\int_{\FBR} \frac{(1+|\xi|^2)^{N/2}}{(1+|\xi|^2)^{N/2}}(1+s^2)|a|^{-\frac{13}{4}} |\xi_1|^{\alpha-1}|\FBcF \Phi(\mathbf{\xi},s,a)|\,\,|\FBcF \chi_1(-a\xi_1)|\\
&\left|\phi_2\left(\frac{\xi_2/\xi_1-s}{|a|^{1/2}}\right)\right|
\frac{ \D s \D a }{(1+s^2)}\D \xi_1\D \xi_2\\
&\lesssim \rho_{N+\alpha-1,N,2,\frac{13}{4}} ^{0,0,0,0} (\FBcF\Phi)+\rho_{N+\alpha-1,N,2,0} ^{0,0,0,0} (\FBcF\Phi),
\end{align*}
where $N\in\FBN$, $N>2$ and $\epsilon>0$.
The terms of the form $|\partial_{x_2}^{\beta}(S^t_{\psi}\Phi)(x_1,x_2)|$, $\beta\in  \mathbb{Z}_+$, can be estimated in a similar fashion.
\par
Next we consider multiplications by  $x_1 ^k$,  $k\in  \mathbb{Z}_+$.
By formula \eqref{FBfrequencysynthesisoperator}, we have that
\begin{align*}
&|x_1^{k}(S^t_{\psi}\Phi)(x_1,x_2)|\\
&=|\int_{\FBR^2}x_1^ke^{2\pi i x\cdot\xi} \int_{\FBR^{\times}}\int_{\FBR} |a|^{-1/4} \FBcF \Phi(\mathbf{\xi},s,a)\,\, |\xi_1|^{-1}\times\\
&\times\overline{\FBcF \chi_1(-a\xi_1)\phi_2\left(\frac{\xi_2/\xi_1-s}{|a|^{1/2}}\right)}
\frac{ \D s \D a }{|a|^3}\D \xi_1\D \xi_2|\\
&=|\int_{\FBR^2}(2\pi i)^{-k}e^{2\pi i x\cdot\xi} \int_{\FBR^{\times}}\int_{\FBR} |a|^{-1/4} \partial_{\xi_1}^{k}[\FBcF \Phi(\mathbf{\xi},s,a)\,\, |\xi_1|^{-1}\times\\
&\times\overline{\FBcF \chi_1(-a\xi_1)\phi_2\left(\frac{\xi_2/\xi_1-s}{|a|^{1/2}}\right)}]
\frac{ \D s \D a }{|a|^3}\D \xi_1\D \xi_2|\\
&\lesssim \int_{\FBR^2}\int_{\FBR^{\times}}\int_{\FBR} |a|^{-1/4} |\partial_{\xi_1}^{k}[\FBcF \Phi(\mathbf{\xi},s,a)\,\, |\xi_1|^{-1}\times\\
&\times\overline{\FBcF \chi_1(-a\xi_1)\phi_2\left(\frac{\xi_2/\xi_1-s}{|a|^{1/2}}\right)}]|
\frac{ \D s \D a }{|a|^3}\D \xi_1\D \xi_2,
\end{align*}which is less than or equal to a finite sum of addends of the form
\begin{align*}
&\int_{\FBR^2}\int_{\FBR^{\times}}\int_{\FBR} |a|^{-\frac{1}{4}+k_3-\frac{k_4}{2}} |\partial_{\xi_1}^{k_1}\FBcF \Phi(\mathbf{\xi},s,a)| |\xi_1|^{-k_2-k_4}|(\FBcF \chi_1)^{(k_3)}(-a\xi_1)|\times\\
&\times|\xi_2|^{k_4}\left|\phi_2^{(k_4)}\left(\frac{\xi_2/\xi_1-s}{|a|^{1/2}}\right)\right|
\frac{ \D s \D a }{|a|^3}\D \xi_1\D \xi_2\\
&=\int_{\FBR^2}\int_{\FBR^{\times}}\int_{\FBR} |a|^{-\frac{13}{4}+k_3-\frac{k_4}{2}} (1+|\xi|^2)^{N/2}(1+s^2)|\partial_{\xi_1}^{k_1}\FBcF \Phi(\mathbf{\xi},s,a)| |\xi_1|^{-k_2-k_4}\times\\
&\times|(\FBcF \chi_1)^{(k_3)}(-a\xi_1)||\xi_2|^{k_4}\left|\phi_2^{(k_4)}\left(\frac{\xi_2/\xi_1-s}{|a|^{1/2}}\right)\right|
\frac{ \D s \D a }{(1+s^2)(1+|\xi|^2)^{N/2}}\D \xi_1\D \xi_2,
\end{align*}
where $k_1,\, k_2,\, k_3,\, k_4\in\FBN$ are less than $k$.
\par
Since $\FBcS_0(\FBR)$ is closed under multiplications by a polynomial, by Lemma \ref{FBLm:mainlm} it follows that
for any given $k_3,\, m\in\FBN$ there exists $g\in  \mathcal{S}_{0} (\mathbb{R})$ such that
\begin{equation*}
(\FBcF \chi_1)^{(k_3)}(-a\xi_1)=-a^m \xi_1^m \FBcF g (-a\xi_1), \;\;\; \xi_1 \in \mathbb{R},\, a\in\FBR^{\times},
\end{equation*}
and we can continue the chain of inequalities with terms of the form
\begin{align*}
&\int_{\FBR^2}\int_{\FBR^{\times}}\int_{\FBR} |a|^{-\frac{13}{4}+k_3-\frac{k_4}{2}+m} (1+|\xi|^2)^{N/2}(1+s^2)|\partial_{\xi_1}^{k_1}\FBcF \Phi(\mathbf{\xi},s,a)| |\xi_1|^{-k_2-k_4+m}\\
&|\FBcF g(-a\xi_1)||\xi_2|^{k_4}\left|\phi_2^{(k_4)}\left(\frac{\xi_2/\xi_1-s}{|a|^{1/2}}\right)\right|
\frac{ \D s \D a }{(1+s^2)(1+|\xi|^2)^{N/2}}\D \xi_1\D \xi_2.
\end{align*}
Finally, choosing $N\in\FBN$, $N>2$, $m\geq k_2+k_4$ and splitting the integral over $\FBR^{\times}$ into integrals over
$|a|<\epsilon$ and $|a|>\epsilon$, with $\epsilon>0$, we obtain that
\begin{align*}
&\int_{\FBR^2}\int_{\FBR^{\times}}\int_{\FBR} |a|^{-\frac{13}{4}+k_3-\frac{k_4}{2}+m} (1+|\xi|^2)^{N/2}(1+s^2)|\partial_{\xi_1}^{k_1}\FBcF \Phi(\mathbf{\xi},s,a)| |\xi_1|^{-k_2-k_4+m}\\
&|\FBcF g(-a\xi_1)||\xi_2|^{k_4}\left|\phi_2^{(k_4)}\left(\frac{\xi_2/\xi_1-s}{|a|^{1/2}}\right)\right|
\frac{ \D s \D a }{(1+s^2)(1+|\xi|^2)^{N/2}}\D \xi_1\D \xi_2\\
&\lesssim [\rho_{N+m-k_2-k_4,N+k_4,2,|\frac{13}{4}-k_3+\frac{k_4}{2}-m|} ^{k_1,0,0,0} (\FBcF\Phi)
+\rho_{N+m-k_2-k_4,N+k_4,2,|-k_3+\frac{k_4}{2}-m|} ^{k_1,0,0,0} (\FBcF\Phi)],
\end{align*}
which is dominated by a single seminorm. We can treat $|{x_2}^{k}(S^t_{\psi}\Phi)(x_1,x_2)|$, $k\in  \mathbb{Z}_+$, in the same manner
and we conclude that the shearlet synthesis $\mathcal S_\psi^t $ is  a continuous map from
$  \mathcal{S}(\mathbb {S})$ into $ \mathcal{S}(\mathbb{R}^2) $.
Finally, it remains to prove that $S^t_{\psi}\Phi\in\FBcS_0(\FBR^2)$. The idea is to prove the equivalent condition
\begin{equation*}
\lim_{\xi\to0}\frac{\FBcF S^t_{\psi}\Phi(\xi)}{|\xi|^k}=0,
\end{equation*}
for every $k\in\FBN$, see {\cite[Lemma 6.0.4]{FBhol1995}}. We refer to \cite{FBbartoluccithesis} for the details.
\end{proof}
\section{The shearlet transform on $\mathcal{S}'_0(\FBR^2)$}%%%%%%%%%%%%%%%%%%%%%%%%%%%%%%	THE SHEARLET TRANSFORM OF DISTRIBUTIONS
\label{FBsec:5}
In this last section, we extend the definition of the shearlet transform to the space of Lizorkin distributions and we show that our definition extends the ones introduced so far. In particular, we prove its consistency with the classical definition
for test functions.
\par
We recall that we consider admissible vectors $\psi$ of the form \eqref{FBeqn:shearletfactorization2} with $\chi_1$ defined by \eqref{FBeqn:phi} in $\mathcal{S}_0(\FBR)$ and $\phi_2=\mathcal{F}\psi_2\in \FBcS(\FBR)$.
\begin{definition}[\cite{FBbartoluccithesis}]
We define the shearlet transform of $f\in \FBcS'_0(\FBR^2)$ with respect to $\psi$ as follows
\begin{equation*}
( \FBcS_{\psi}f,\Phi ) = ( f,\FBcS^t_{\overline{\psi}}\Phi ),\quad \Phi\in \FBcS(\mathbb{S}).
\end{equation*}
\label{FBdefn:shearlettransformdistributions}
\end{definition}
The consistency of Definition~\ref{FBdefn:shearlettransformdistributions} is guaranteed by Theorem~\ref{FBthm:continuitysynthesisoperator}.
Furthermore, it follows straightforwardly that the shearlet transform of $f\in \FBcS'_0(\FBR^2)$ is a well defined distribution in $ \FBcS'(\mathbb{S})$.
\begin{proposition}[\cite{FBbartoluccithesis}]
The shearlet transform $\FBcS_{\psi}$ given by Definition \ref{FBdefn:shearlettransformdistributions}
is a continuous and  linear map from $\FBcS'_0(\FBR^2) $ into $ \FBcS'(\mathbb{S})$.
\end{proposition}

\par

The next theorem shows that Definition~\ref{FBdefn:shearlettransformdistributions} is in fact
consistent with the definition for test functions (Definition \ref{FBdefsheartransf})
and  that it generalizes the extension considered in \cite{FBkula09,FBgr11} where the shearlet transform
of a tempered distribution $f$ with respect to an admissible vector $\psi\in\FBcS(\FBR^2)$ is given by the function
\begin{equation*}
\FBcS_{\psi}f(b,s,a)={_{\FBcS'(\FBR^2)}}( f, S_{b,s,a}\psi)_{\FBcS(\FBR^2)},
\end{equation*}
for every $(b,s,a)\in\mathbb{S}$.
Following the coorbit space approach, given a suitable test function space usually denoted by $\mathcal{H}_{1,w}$, where $w$ is a weight function,
and its anti-dual $\mathcal{H}_{1,w}^{\sim}$, the extended shearlet transform of $f\in \mathcal{H}_{1,w}^{\sim}$ with respect to $\psi\in \mathcal{H}_{1,w}$ is defined by
\[
\FBcS_{\psi}f(b,s,a)={_{\mathcal{H}_{1,w}^{\sim}}}(f, S_{b,s,a}\psi)_{\mathcal{H}_{1,w}},
\]
for every $(b,s,a)\in\mathbb{S}$, \cite{dahlikeetal}. Theorem~\ref{FBteo:desingularization} shows the equivalence of our duality approach with the coorbit space one.
Precisely, Theorem~\ref{FBteo:desingularization} states that the the shearlet transform of any Lizorkin distribution is given
by the function defined as
\[(b,s,a)\mapsto{_{\FBcS'_0(\FBR^2)}}( f, S_{b,s,a}\psi)_{\FBcS_0(\FBR^2)},\]
for every $(b,s,a)\in\mathbb{S}$.

\begin{theorem}[\cite{FBbartoluccithesis}]
Let $f\in \FBcS_0'(\FBR^2)$. The shearlet transform of $f$ is given by the function
\begin{equation*}
(b,s,a)\mapsto{_{\FBcS'_0(\FBR^2)}}( f, S_{b,s,a}\psi)_{\FBcS_0(\FBR^2)},
\end{equation*}
that is,
\begin{equation*}
(\mathcal{S}_{\psi}f,\Phi)= \int_{\mathbb{S}} ( f, S_{b,s,a}\psi)\Phi(b,s,a) {\rm d}\mu(b,s,a), \;\;\;
\Phi\in\FBcS(\mathbb{S}).
\end{equation*}
\label{FBteo:desingularization}
\label{thm:distributioncoorbit}
\end{theorem}
\begin{proof}
Consider $f\in\FBcS'_0(\FBR^2)$. Since the space of Lizorkin distributions $\FBcS'_0(\FBR^2)$ is canonically isomorphic
to the quotient of $\FBcS'(\FBR^2)$ by the space of polynomials, by Schwartz' structural theorem \cite[Theor\'em VI]{FBLS1951},
we can write $f=g^{(\alpha)}+p$, where $g$ is a continuous slowly growing function, $\alpha\in\FBN^2$ and $p$ is a polynomial. Then, for any $\Phi\in\FBcS(\mathbb{S})$
\begin{align*}
( g^{(\alpha)},\mathcal{S}_{\overline{\psi}}^t\Phi)
&=(-1)^{|\alpha|}( g,\mathcal{S}_{\overline{\psi}}^t\Phi^{(\alpha)})=(-1)^{|\alpha|}\int_{\FBR^2} g(x) \mathcal{S}_{\overline{\psi}}^t\Phi^{(\alpha)}(x)\D x\\
&=(-1)^{|\alpha|} \int_{\FBR^2} g(x) \int_{\mathbb{S}} \Phi(b,s,a)\,\, ({S}_{b,s,a}\psi)^{(\alpha)}(x)
\,\,{\rm d}\mu(b,s,a) \D x\\
&=\int_{\mathbb{S}} \Phi(b,s,a)\,\,(-1)^{|\alpha|}\int_{\FBR^2} g(x)({S}_{b,s,a}\psi)^{(\alpha)}(x)\D x {\rm d}\mu(b,s,a)\\
&=\int_{\mathbb{S}} \Phi(b,s,a)\,\,(-1)^{|\alpha|}( g,({S}_{b,s,a}\psi)^{(\alpha)}) {\rm d}\mu(b,s,a)\\
&=\int_{\mathbb{S}} \Phi(b,s,a)\,\,( g^{(\alpha)},{S}_{b,s,a}\psi) {\rm d}\mu(b,s,a).
\end{align*}
Analogously, we have that
\begin{align*}
( p,\mathcal{S}_{\overline{\psi}}^t\Phi)= \int_{\mathbb{S}} \Phi(b,s,a)\,\,( p,{S}_{b,s,a}\psi) {\rm d}\mu(b,s,a).
\end{align*}
Therefore, we obtain
\begin{align*}
( f,\mathcal{S}_{\overline{\psi}}^t\Phi)=( g^{(\alpha)}+p,\mathcal{S}_{\overline{\psi}}^t\Phi)
&=\int_{\mathbb{S}} \Phi(b,s,a)\,\,( g^{(\alpha)}+p,{S}_{b,s,a}\psi) {\rm d}\mu(b,s,a)\\
&=\int_{\mathbb{S}} \Phi(b,s,a)\,\,( f,{S}_{b,s,a}\psi) {\rm d}\mu(b,s,a),
\end{align*}
which concludes the proof.
\end{proof}
%%%%%%%%%
%
%\textcolor{red}{As mentioned before, the authors in \cite{dahlikeetal} extend the shearlet transform following the coorbit space theory and we %provide an extension following the duality approach
%of Schwartz. Theorem~\ref{FBteo:desingularization} shows in a certain sense the equivalence of the two approaches.}
%
%%%%%%%%%
\section*{Appendix}
\label{FB:appendix}
\addcontentsline{toc}{section}{Appendix}
As mentioned in the introduction, both the ridgelet transform and the shearlet transform are related to the wavelet and the Radon transforms and, as we now show, they are
related as well. The Appendix is devoted to prove this connection, which has in part inspired our work. We start briefly recalling the ridgelet transform and we refer to \cite{FBcado99} as a classical reference.
\par
We fix $\psi\in\FBcS(\FBR)$ and we define for every $(\theta, b, a)\in[-\pi,\pi)\times\FBR\times\FBR_+$ the function $R_{\theta, b, a}\psi$ as
\begin{equation*}
R_{\theta, b, a}\psi(x)=\frac{1}{a}\psi\left(\frac{x\cdot n(\theta)-b}{a}\right),\qquad x\in\FBR^2,
\end{equation*}
where $n(\theta)=(\cos{\theta},\sin{\theta})$. Then, the ridgelet transform of $f\in L^1(\FBR^2)$ with respect to $\psi$ is given by
\begin{equation*}
\FBcR_{\psi}f(\theta,b,a)=\int_{\FBR^2} f(x)\overline{R_{\theta, b, a}\psi(x)}\, \D x,
%= ( f(x),\overline{\psi_{u,b,a}} (x) ),
\;\;\; (\theta,b,a)\in[-\pi,\pi)\times\FBR\times\FBR_+.
\end{equation*}
The ridgelet transform is related to the wavelet transform and the polar Radon transform by the following formula
\begin{equation}\label{FBeqn:rideletwaveletradon}
\FBcR_{\psi}f(\theta,b,a)=\mathcal{W}_{\psi}(\mathcal{R}^{\rm pol}f(\theta, \cdot))(b,a),
\end{equation}
for every $(\theta, b, a)\in[-\pi,\pi)\times\FBR\times\FBR_+$ and where the wavelet transform is one-dimensional and acts on the variable $q$. By equations~\eqref{FBeqn:rideletwaveletradon} and \eqref{FBeqn:polaff}, we obtain a relation formula between the ridgelet and the shearlet transform. We consider an admissible vectors $\psi$ of the form \eqref{FBeqn:shearletfactorization2} satisfying conditions \eqref{FBeq:4} and with $\chi_1$ defined by \eqref{FBeqn:phi} belonging to $\mathcal{S}(\FBR)$.
\begin{proposition}\label{FBprop:rideletwaveletradon}
For any $f\in L^1(\FBR^2)\cap L^2(\FBR^2)$ and $((b_1,b_2),s,a)\in \FBR^2\times\FBR\times\FBR_+$,
\begin{align*}
&\mathcal{S}_{\psi}f((b_1,b_2),s,a)\\
&=|a|^{-\frac{3}{4}}\int_{\mathbb{R}}\frac{1}{\sqrt[4]{1+v^2}}\FBcR_{\chi_1}f\left(\arctan{v},\frac{b_1+vb_2}{\sqrt{1+v^2}},\frac{a}{\sqrt{1+v^2}}\right)\overline{\phi_2\left(\frac{v-s}{|a|^{1/2}}\right)}\ {\rm d}v.
\end{align*}
\end{proposition}
\begin{proof}
By Theorem~\ref{FBteo:teocentrale}, for any $f\in L^1(\FBR^2)\cap L^2(\FBR^2)$ and $((b_1,b_2),s,a)\in \FBR^2\times\FBR\times\FBR^{\times}$, the shearlet transform has the following expression
\begin{equation}\label{FBeqn:firstapositive}
\mathcal{S}_{\psi}f((b_1,b_2),s,a)=|a|^{-\frac{3}{4}}\int_{\mathbb{R}}\mathcal{W}_{\chi_1}(\FBcR^{\rm aff} f(v,\cdot))(b_1+vb_2,a)\overline{\phi_2\left(\frac{v-s}{|a|^{1/2}}\right)}\ {\rm d}v.
\end{equation}
By equation \eqref{FBeqn:polaff}, for every $v\in\FBR$, $(b_1,b_2)\in\FBR^2$ and $a\in\FBR_+$, we compute
\begin{align}\label{FBeqn:rideletwavelet}
\nonumber&\mathcal{W}_{\chi_1}(\FBcR^{\rm aff} f(v,\cdot))(b_1+vb_2,a)\\
\nonumber&=\frac{1}{\sqrt{1+v^2}}\mathcal{W}_{\chi_1}(\FBcR^{\rm pol} f(\arctan{v},\frac{\cdot}{\sqrt{1+v^2}}))(b_1+vb_2,a)\\
\nonumber&=\frac{1}{\sqrt[4]{1+v^2}}\mathcal{W}_{\chi_1}(D_{\sqrt{1+v^2}}\FBcR^{\rm pol} f(\arctan{v},\cdot))(b_1+vb_2,a)\\
\nonumber&=\frac{1}{\sqrt[4]{1+v^2}}\mathcal{W}_{\chi_1}(\FBcR^{\rm pol} f(\arctan{v},\cdot))\left(\frac{b_1+vb_2}{\sqrt{1+v^2}},\frac{a}{\sqrt{1+v^2}}\right)\\
&=\frac{1}{\sqrt[4]{1+v^2}}\mathcal{R}_{\chi_1}f\left(\arctan{v},\frac{b_1+vb_2}{\sqrt{1+v^2}},\frac{a}{\sqrt{1+v^2}}\right).
\end{align}
Replacing formula~\eqref{FBeqn:rideletwavelet} in \eqref{FBeqn:firstapositive} we obtain the desired relation.
\end{proof}
\section*{Acknowledgement}
F. Bartolucci is part of the Computational Harmonic Analysis \& Machine Learning unit of the Machine Learning Genoa Center (MalGa). S. Pilipovi\'{c} and N. Teofanov were supported by the
Ministry of Education, Science and Technological Development of the
Republic of Serbia through Project 174024.

\end{document}